\documentclass[11pt]{article}

\usepackage[utf8]{inputenc}
\usepackage[T1]{fontenc}
\usepackage[english]{babel}
\usepackage{amsmath,amssymb,amsthm,mathtools}
\usepackage{geometry}
\usepackage[expansion=false]{microtype}
\usepackage{hyperref}
\usepackage{enumitem}
\usepackage{authblk}

\hypersetup{
    colorlinks=true,
    linkcolor=black,
    citecolor=black,
    urlcolor=blue
}

\numberwithin{equation}{section}

\theoremstyle{plain}
\newtheorem{theorem}{Theorem}[section]
\newtheorem{corollary}[theorem]{Corollary}
\newtheorem{lemma}[theorem]{Lemma}
\newtheorem{proposition}[theorem]{Proposition}

\theoremstyle{remark}
\newtheorem{remark}[theorem]{Remark}

\theoremstyle{definition}
\newtheorem{definition}[theorem]{Definition}

\newcommand{\R}{\mathbb R}
\newcommand{\norm}[1]{\left\lVert #1\right\rVert}
\newcommand{\cR}{\mathcal R}

\title{An abstract averaging method for quasilinear equations}

\author[1]{Jean Mawhin}
\author[2]{Jorge Novoa}

\affil[1]{Institut de Recherche en Math\'ematique et Physique,
Universit\'e Catholique de Louvain,
Chemin du Cyclotron 2, 1348 Louvain-la-Neuve, Belgium\\
\texttt{jean.mawhin@uclouvain.be}}

\affil[2]{Department of Mathematical Engineering,
Universidad de Chile,
Beauchef 851, Santiago, Chile\\
\texttt{jnovoa@dim.uchile.cl}}

\providecommand{\backmatter}{}

\date{}

\begin{document}

\maketitle

\begin{abstract}
We develop an abstract averaging method for equations
\(Mx=\varepsilon N(x,\varepsilon)\) in Banach spaces. The method applies when the equation can be separated into
an averaged compatibility condition and a nonlinear complementary problem, and it allows the principal operator $M$ to be nonlinear. We prove
that every nondegenerate zero of the corresponding averaged equation generates,
for all sufficiently small nonzero values of $\varepsilon$, a nearby solution.
Under a suitable small Lipschitz condition this solution is locally unique,
while in the differentiable case it belongs to a unique local $C^1$ branch.

Our main result can be applied to the abstract linear Fredholm averaging theory, as a special case, but also covers quasilinear problems for which the principal operator has no invertible linearization. We also
determine the asymptotic profile of the solutions. For a vector $p$-Laplacian Duffing system, with our results, we can prove the existence of multiple periodic solutions and an explicit response scale of order $|\varepsilon|^{1/(p-1)}$. Our second application consists on a nonsmooth equation in a Hilbert space.

\noindent {\em AMS Classification (MSC 2020).} 34C29, 34C25, 47H10
\end{abstract}

\section{Introduction}

Averaging methods relate small-parameter periodic problems to a simpler equation on the set of unperturbed solutions. For the classical first-order system
\begin{equation}\label{eq:classical}
    x'=\varepsilon f(t,x,\varepsilon),
    \qquad x(0)=x(T),
\end{equation}
the relevant equation is
\begin{equation}\label{eq:classical-bif}
    \int_0^T f(t,a,0)\,dt=0.
\end{equation}
Under a nondegeneracy condition for $a^*$, a zero of
\eqref{eq:classical-bif}, there exists $T$-periodic solutions of
\eqref{eq:classical} for small $|\varepsilon|$ in a neighborhood of $a^*$. This classical principle can
be proved in several ways, including the Poincar\'e operator and the
Cesari--Hale approach; see \cite{RoM,Sanders2007} and the references therein.
\medskip

Developments in averaging methods has been developed in several (and complementary) directions. Buic\u{a} and Llibre
\cite{BuicaLlibre2004} obtained existence results for periodic orbits through Brouwer degree under weaker regularity assumptions. Higher-order
averaging and refinements of the classical existence criteria were developed
in \cite{LlibreNovaesTeixeira2014,LlibreNovaes2015}, while C\^{a}ndido,
Llibre and Novaes \cite{CandidoLlibreNovaes2017} combined averaging with Lyapunov--Schmidt reduction to study the persistence of periodic solutions.
In more recent results, C\^{a}ndido and Novaes \cite{CandidoNovaes2026} investigated
the stability of smooth branches obtained by this type of reduction.
\medskip

These contributions concern periodic dynamics and reductions under perturbations, the main difficulty  treated here is different, namely, the nonlinearity of
the principal operator and its complementary equation.
\medskip

A useful way to understand the averaging problem is to separate the periodic problem into two parts. For $\varepsilon\ne0$, \eqref{eq:classical} is equivalent to
\[
    x'=\varepsilon f(t,x,\varepsilon),
    \qquad
    \int_0^T f(t,x(t),\varepsilon)\,dt=0.
\]
The first equation the oscillatory component, while the second equation is
the compatibility condition that selects the correct constant state when
$\varepsilon=0$, since if $\varepsilon = 0$ then $x' = 0$ and hence $x \equiv c$ for some constant. 

\medskip

This point of view leads naturally to an implicit function formulation (see \cite{Maw1, RoM}). In a more recent paper,  Cid, Mawhin and Zima \cite{CidMawhinZima2021}
developed an abstract version for equations
\[
    Lx=\varepsilon N(x,\varepsilon),
\]
where $L$ is a Fredholm operator of index zero. In that case, one can derive some differentiability conditions since the operator $L$ is linear, and hence the implicit function theorem can be applied.
\medskip

This linear decomposition is useful, but it also points to the main difficulty for quasilinear problems. If the principal operator is nonlinear,  we can not (in general) conclude that there exists a linear inverse on the complementary space, and, in fact, no invertible linearization at all. The averaged equation can still be
identified, but one must solve the complementary part without replacing the
nonlinear operator by a linear one.
\medskip

In particular, our problems of interest for which this exact problem appears are the periodic differential equations of the form
\begin{equation}\label{eq:phi-intro}
    \bigl(\phi(u')\bigr)'
    =\varepsilon f(t,u,u',\varepsilon),
\end{equation}
which include $p$-Laplacian and $p$-Laplacian-like operators, relativistic
acceleration and prescribed mean curvature operators. It 
has been extensively studied using topological and variational methods; see, for
example,
\cite{ManasevichMawhin1998,ManasevichMawhin2000,Mawhin2001,Maw2,GMMT}.
\medskip

In \cite{MawhinNovoaIFT}, we treated
\eqref{eq:phi-intro} by a direct implicit function argument. That approach
requires $\phi^{-1}$ to be $C^1$ near the origin and
\[
    D\phi^{-1}(0)
\]
to be a continuous isomorphism. It therefore covers, among others, the classical second-order and relativistic cases and gives a unique local branch
near each nondegenerate zero of the averaged equation.
\medskip

The problem with this method, as described before, is already visible for the $p$-Laplacian operator,
\[
    \Phi_p(\xi)=|\xi|^{p-2}\xi,
    \qquad
    \Phi_p^{-1}=\Phi_{p'},
    \qquad
    p'=\frac{p}{p-1}.
\]
If $p>2$, then $\Phi_{p'}$ is not differentiable at the origin, if
$1<p<2$, its derivative there is degenerate. Thus, except for $p=2$, the
complementary equation cannot be recovered from an invertible linearization
of the principal part. Notice that this is not a difficulty with the averaged equation, the obstruction lies in the complementary problem.
\medskip

The loss of linearization also changes the size of the response, for example, in the
$p$-Laplacian case, the inverse $\Phi_{p'}$ is positively homogeneous of
degree $1/(p-1)$. A forcing of size $|\varepsilon|$ therefore produces a
complementary component of size $|\varepsilon|^{1/(p-1)}$. For $p>2$ this
is larger than the classical $O(|\varepsilon|)$ scale, so a first-order
linear expansion cannot capture the leading behavior.
\medskip

There are two main contributions. First, the
abstract Fredholm theory of \cite{CidMawhinZima2021} is contained in our
framework, since if the principal operator is Fredholm of index zero, the usual inverse on a complement of the kernel produces the complementary equation, and the
$C^1$ result of our main theorem recovers the corresponding local averaging result. Second, the same reduction still works when the principal operator
is genuinely nonlinear and such a linear inverse is no longer available.
This allows us to treat quasilinear equations in Banach spaces, and also classical problems as, for example, the $p$-Laplacian differential equations that lie outside both the Fredholm framework and the smooth (or rather at least differenatiable) setting of
\cite{MawhinNovoaIFT}. To the best of our knowledge, the abstract averaging results cited above do not cover this kind of quasilinear problems at this level of generality.
\medskip

We consider
\begin{equation}\label{eq:main}
    Mx=\varepsilon N(x,\varepsilon),
    \qquad 0<|\varepsilon|\ll1,
\end{equation}
with the following decomposition
\[
    X=E\oplus X_2,\qquad
    Z=Z_1\oplus Z_2,\qquad
    E=\ker M.
\]
Instead of linearizing $M$ on $X_2$, we solve the complementary equation
directly. This leads to the exact local system
\[
    QN(a+v,\varepsilon)=0,
    \qquad
    v=R(a+v,\varepsilon),
    \qquad a\in E,\quad v\in X_2,
\]
where the first equation is the averaged compatibility condition and the
second determines the complementary component. Nonlinear complementary
inverses and quasilinear continuation methods have appeared before in
\cite{ManasevichMawhin1998,GeRen2004}; here they are used to build a local
averaging principle near a nondegenerate zero of the averaged map.
\medskip

The resulting theorem is composed from three distinct results. Firstly, under a compactness assumption we can prove existence. Secondly, under a small Lipschitz estimate we can prove a little further, there is local existence and
uniqueness. At last, if the operator from the  complementary problem is differentiable, we can prove the existence of a unique local $C^1$ branch. 

\medskip

The same reduction also gives
the asymptotic profile. In particular, a nonlinear inverse that is homogeneous of degree $\gamma$ produces the scale
$|\varepsilon|^\gamma$ instead of the classical $O(\varepsilon)$ scale.
\medskip

In the last section we show two applications. The first one regards a vector $p$--Laplacian Duffing system with $p>2$, for this system we show that there exists multiple periodic solutions for sufficiently small $|\varepsilon|$.  The second application is a periodic infinite-dimensional problem in a Hilbert space in which the inverse is nonsmooth but Lipschitz. In this problem, we not only show existence of periodic solution for any sufficiently small $|\varepsilon|$ but also prove uniqueness.
\medskip

The paper is organized as follows. Section \ref{sec:abstract} gives the
abstract reduction, the three local averaging principles, and the asymptotic
profile. Section \ref{sec:example} applies the theory to the vector
$p$-Laplacian Duffing system and to the Hilbert-space problem.

\section{A nonlinear averaging principle}\label{sec:abstract}

Let $X$ and $Z$ be
real Banach spaces, and let $M:D(M)\subset X\to Z$ be an operator. Suppose
that
\begin{equation}\label{eq:splitting}
 E:=\ker M,\qquad X=E\oplus X_2,\qquad Z=Z_1\oplus Z_2,
 \qquad \operatorname{Im}M\subset Z_2,
\end{equation}
where all the spaces appearing in the decompositions are closed. Denote
by $P:X\to E$ and $Q:Z\to Z_1$ the corresponding projections. Let
$U\subset X$ be open, and let
$N:\overline U\times[-\varepsilon_1,\varepsilon_1]\to Z$ be continuous.
\medskip

Write $x=a+v$, with $a=Px\in E$ and $v=(I-P)x\in X_2$. Since $QM=0$,
for $\varepsilon\ne0$ the equation $Mx=\varepsilon N(x,\varepsilon)$
is equivalent to
\begin{equation}\label{eq:split-equations}
 \begin{cases}
 QN(a+v,\varepsilon)=0,\\
 M(a+v)=\varepsilon(I-Q)N(a+v,\varepsilon).
 \end{cases}
\end{equation}
The first equation is a compatibility condition; the second determines
the complementary component $v$. For solutions with $v\to0$ as $\varepsilon\to0$, the compatibility
condition is equivalent to $F(a)=0$, where $F$ is defined as
\[
 F:E\cap U\longrightarrow Z_1,\qquad F(a):=QN(a,0).
\]
Notice that if we fix $a_0\in E\cap U$ and suppose that $QN$ is $C^1$ near $(a_0,0)$ and
\begin{equation}\label{eq:ND}
 F(a_0)=0,\qquad
 F'(a_0)=D_x(QN)(a_0,0)|_E:E\to Z_1
 \quad\text{is an isomorphism},
\end{equation}
then we could use (in some sense) the implicit function theorem for the compatibility condition near a neighborhood of $a_0$.  Thus the nondegeneracy condition will be imposed only on the averaged map $F$.
\medskip

The main issue is how to solve the second equation in
\eqref{eq:split-equations} without a linear inverse on $X_2$, which is the main motivation for the following definition.

\begin{definition}[Complementary correction]\label{def:correction}
A continuous map
$R:\overline U\times[-\varepsilon_1,\varepsilon_1]\to X_2$ with
$R(x,0)=0$ is called a \emph{complementary correction} if it verifies the following properties,
\begin{enumerate}
    \item For every
$(x,\varepsilon)$ in its domain,
$\varepsilon(I-Q)N(x,\varepsilon)\in\operatorname{Im}M$,
$Px+R(x,\varepsilon)\in D(M)$, and
\begin{equation}\label{eq:correction}
 M\bigl(Px+R(x,\varepsilon)\bigr)
 =\varepsilon(I-Q)N(x,\varepsilon).
\end{equation}
\item Every solution $x$ of \eqref{eq:main} satisfies
$$R(x,\varepsilon)=(I-P)x.$$
\end{enumerate}

\end{definition}

The corrector $R$ is not a global inverse for $M$, rather than that, it only gives a complementary component that is compatible with the prescribed kernel $Px$.

\medskip

An important observation is the following. When $M$ is invariant under translations by
$E$, in the sense that $M(a+v)=M(v)$ whenever the terms are defined, and the restriction of $M$ to $D(M)\cap X_2$ is bijective onto $Z_2$ with
continuous inverse $\cR:Z_2\to D(M)\cap X_2$, then one may take
\begin{equation}\label{eq:typical-correction}
 R(x,\varepsilon)
 =\cR\bigl(\varepsilon(I-Q)N(x,\varepsilon)\bigr).
\end{equation}
Indeed, translation invariance gives \eqref{eq:correction}, while
injectivity on $X_2$ guarantees the last assertion in the definition.
The theorem below, however, only requires a complementary correction;
it does not require this particular representation.
\medskip

For a simple example, let $X=Z=\R^2$, write $x=(a,v)$, and set
\[
 M(a,v)=(0,v^3).
\]
Here $E=\R\times\{0\}$, $X_2=\{0\}\times\R$, and we may choose
$Z_1=\R\times\{0\}$, $Z_2=\{0\}\times\R$. Writing $N=(N_1,N_2)$,
the original equation becomes, for $\varepsilon\ne0$,
\[
 N_1(a,v,\varepsilon)=0,\qquad
 v^3=\varepsilon N_2(a,v,\varepsilon).
\]
The complementary equation can therefore be solved exactly by
\[
 R(x,\varepsilon)
 =\left(0,\sqrt[3]{\varepsilon N_2(x,\varepsilon)}\right).
\]
Although the derivative of $v\mapsto v^3$ vanishes at $v=0$, its
nonlinear inverse is continuous. Moreover, if
$N_2((a_*,0),0)\ne0$, the complementary component of a solution near
$(a_*,0)$ has size $|\varepsilon|^{1/3}$, rather than the usual
$|\varepsilon|$.
\medskip

The next lemma turns the second equation in \eqref{eq:split-equations} into a fixed-point
problem. With this, we will only use the implicit function theorem for the compatibility equation, where the nondegeneracy condition
\eqref{eq:ND} is available.

\begin{lemma}\label{lem:reduction}
Assume that $QN$ is $C^1$ near $(a_0,0)$ and that \eqref{eq:ND} holds.
Let $R$ be a complementary correction as defined in \ref{def:correction}, and write $x=a+v$, with $a\in E$
and $v\in X_2$. For $\varepsilon\ne0$,
\begin{equation}\label{eq:reduction}
 Mx=\varepsilon N(x,\varepsilon)
 \quad\Longleftrightarrow\quad
 \begin{cases}
 QN(a+v,\varepsilon)=0,\\
 v=R(a+v,\varepsilon).
 \end{cases}
\end{equation}
Moreover, near $(a_0,0)$ there exists a $C^1$ map $A(v,\varepsilon)$, with
$A(0,0)=a_0$, such that $$QN(a+v,\varepsilon)=0 \quad\Longleftrightarrow\quad a=A(v,\varepsilon).$$

Hence local solutions are exactly the fixed points of
\begin{equation}\label{eq:T}
 T_\varepsilon(v):=R(A(v,\varepsilon)+v,\varepsilon).
\end{equation}
At last, by setting
$B_0:=F'(a_0)^{-1}D_x(QN)(a_0,0)|_{X_2}\in\mathcal L(X_2,E)$, then
$D_vA(0,0)=-B_0$, and for every $\eta>0$ (after shrinking the neighborhood if necessary),
\begin{equation}\label{eq:reduced-Lip}
 \norm{A(v_1,\varepsilon)-A(v_2,\varepsilon)+v_1-v_2}_X
 \le(1+\norm{B_0}+\eta)\norm{v_1-v_2}_X.
\end{equation}
\end{lemma}

\begin{proof}
Since $QM=0$, applying $Q$ to the original equation gives
$QN(x,\varepsilon)=0$ for $\varepsilon\ne0$. Then, since $R$ is a complementary correction, we can conclude that $v=R(x,\varepsilon)$. 

\medskip

Conversely, if the two equations on
the right-hand side of \eqref{eq:reduction} hold, then
\eqref{eq:correction} gives
\[
 M(a+v)=\varepsilon(I-Q)N(a+v,\varepsilon)
         =\varepsilon N(a+v,\varepsilon),
\]
which proves the equivalence.
\medskip

Now consider $\Phi(a,v,\varepsilon)=QN(a+v,\varepsilon)$. Since
$D_a\Phi(a_0,0,0)=F'(a_0)$ is an isomorphism, the implicit function
theorem provides the map $A$, and substitution into
\eqref{eq:reduction} gives \eqref{eq:T}. Differentiating
$\Phi(A(v,\varepsilon),v,\varepsilon)=0$ at $(0,0)$ yields
$D_vA(0,0)=-B_0$. By continuity of $D_vA$, on a sufficiently small
convex neighborhood we have
\[
 \norm{D_vA(v,\varepsilon)+\iota}_{\mathcal L(X_2,X)}
 \le 1+\norm{B_0}+\eta,
\]
where $\iota:X_2\hookrightarrow X$ is the inclusion and $D_vA$ is
viewed as $X$-valued through $E\subset X$.
The mean-value estimate gives \eqref{eq:reduced-Lip}.
\end{proof}

The reduction has two consequences, first, solving the averaged equation
requires only the nondegeneracy of $F$, the possibly nonsmooth principal operator does not enter the implicit function argument. The second implication is the following, the
original problem is locally equivalent to the fixed-point equation $v=T_\varepsilon(v)$ on $X_2$, with $T_0(v)=0$. 

\medskip

We use three alternative
hypotheses on $R$ to obtain a fixed point, uniqueness, or smooth
dependence on the parameter.

\medskip
\noindent\textup{(C) \textbf{(Compactness)}}\quad
$R$ is completely continuous in a neighborhood of $(a_0,0)$.

\medskip
\noindent\textup{(L) \textbf{(Small Lipschitz condition)}}\quad
There exist a neighborhood $V$ of $a_0$ and a function
$\ell:[0,\varepsilon_1]\to[0,\infty)$ such that, for $x,y\in V$,
\[
 \norm{R(x,\varepsilon)-R(y,\varepsilon)}_X
 \le \ell(|\varepsilon|)\norm{x-y}_X,
 \qquad
 \limsup_{r\to0^+}\ell(r)<\frac{1}{1+\norm{B_0}}.
\]

\medskip
\noindent\textup{(D) \textbf{(Differentiability)}}\quad
The map $(x,\varepsilon)\mapsto R(x,\varepsilon)$ is $C^1$
in a neighborhood of $(a_0,0)$ in $X\times\R$.
\medskip

With these definitions, we can finally prove our main result.

\begin{theorem}[Nonlinear averaging principle]\label{thm:main}
Assume that $QN$ is $C^1$ near $(a_0,0)$, that \eqref{eq:ND} holds, and
that $R$ is a complementary correction. Then, for every sufficiently small
$0<|\varepsilon|$, the following hold.
\medskip

\noindent\textup{(i)} Under \textup{(C)}, \eqref{eq:main} has a solution
$x_\varepsilon\to a_0$ in $X$.

\noindent\textup{(ii)} Under \textup{(L)}, \eqref{eq:main} has a unique
solution in a fixed sufficiently small neighborhood of $a_0$, and
$x_\varepsilon\to a_0$ in $X$.

\noindent\textup{(iii)} Under \textup{(D)}, the nearby solutions for
$\varepsilon\ne0$ in equation \eqref{eq:main} lie on a unique local $C^1$ branch
$\varepsilon\mapsto x_\varepsilon$ extending to $x_0=a_0$.
\end{theorem}

\medskip

\begin{proof}
By Lemma~\ref{lem:reduction}, local solutions are exactly the fixed points
of $T_\varepsilon$ in $X_2$.

For \textup{(i)}, choose $r>0$ so that $A$ is defined on
$\overline B_r^{X_2}(0)\times[-\varepsilon_*,\varepsilon_*]$. Since $R$ is
completely continuous, $T_\varepsilon$ is compact. After decreasing
$\varepsilon_*$ if necessary,
\begin{equation}\label{eq:no-boundary}
 v\ne\theta T_\varepsilon(v),\qquad
 \norm v_X=r,\quad 0\le\theta\le1.
\end{equation}
Indeed, otherwise there would exist $\varepsilon_k\to0$,
$\norm{v_k}_X=r$ and $\theta_k\in[0,1]$ with
$v_k=\theta_kT_{\varepsilon_k}(v_k)$. The complete continuity of the operator gives, up to a
subsequence, $v_k\to v$ and $\theta_k\to\theta$. 

\medskip

Passing to the limit and using $R(x,0)=0$ yields $v=0$, contradicting $\norm v_X=r$. Hence
\[
 \deg_{LS}[I-T_\varepsilon,B_r,0]
 =\deg_{LS}[I,B_r,0]=1,
\]
so a fixed point $v_\varepsilon$ exists. The same compactness argument
shows that any sequence of such fixed points with $\varepsilon\to0$ is
relatively compact and converges to $0$, thus
$v_\varepsilon\to0$ and
$x_\varepsilon=A(v_\varepsilon,\varepsilon)+v_\varepsilon\to a_0$.
\medskip

For \textup{(ii)}, using \eqref{eq:reduced-Lip} and \textup{(L)} allow us to
choose $q\in(0,1)$ such that, for small nonzero $\varepsilon$,
\[
 \norm{T_\varepsilon(v_1)-T_\varepsilon(v_2)}_X
 \le q\norm{v_1-v_2}_X.
\]
Moreover,
$T_\varepsilon(0)=R(A(0,\varepsilon),\varepsilon)\to R(a_0,0)=0$, so a
fixed sufficiently small closed ball is invariant under $T_\varepsilon$.
Hence, by Banach fixed point theorem, we can conclude that there exists a unique fixed point, with
\[
 \norm{v_\varepsilon}_X
 \le\frac{\norm{T_\varepsilon(0)}_X}{1-q}\longrightarrow0.
\]
The exact reduction gives existence, local uniqueness and
$x_\varepsilon\to a_0$ for \eqref{eq:main}.
\medskip

To prove \textup{(iii)}, set
\[
 \mathcal G(a,v,\varepsilon)
 :=\bigl(QN(a+v,\varepsilon),\,v-R(a+v,\varepsilon)\bigr).
\]
Since $R(x,0)=0$, one has $D_xR(a_0,0)=0$, and therefore
\[
 D_{(a,v)}\mathcal G(a_0,0,0)
 =
 \begin{pmatrix}
 F'(a_0) & D_x(QN)(a_0,0)|_{X_2}\\
 0 & I
 \end{pmatrix}.
\]
This operator is an isomorphism because $F'(a_0)$ is also an isomorphism, hence, by the implicit function
theorem,  we can conclude that there exists the unique local branch described in \textup{(iii)}.
\end{proof}

\medskip

\begin{remark}[The linear Fredholm case]\label{rem:linear-case}
Notice that when $M=L$ is a 
Fredholm operator of index zero, by choosing
\[
    X=\ker L\oplus X_2,
    \qquad
    Z=Z_1\oplus \operatorname{Im}L,
\]
and letting $K:\operatorname{Im}L\longrightarrow X_2$ be the inverse of $L|_{X_2}$, then its complementary corrector is defined by
\begin{equation}\label{eq:R-linear}
    R(x,\varepsilon)
    =\varepsilon K(I-Q)N(x,\varepsilon).
\end{equation}
If $N$ is $C^1$, then $R$ is $C^1$, and
Theorem~\ref{thm:main}(iii) gives exactly the usual local Fredholm reduction in \cite{CidMawhinZima2021}.

\medskip

Thus the abstract averaging
scheme of \cite{CidMawhinZima2021} is contained in our differentiable case of
Theorem \ref{thm:main}. The new point is that the formula
\eqref{eq:R-linear} is not essential: if $M$ is nonlinear, $K$ may be replaced by a nonlinear map $R$, while the nondegeneracy assumption is still placed only on the averaged map $F$.
\end{remark}

The three alternatives in Theorem~\ref{thm:main} correspond to three different types of regularity of the complementary problem. Compactness is
enough for existence, an small Lipschitz constant turns the reduced equation into a contraction and gives local uniqueness and existence, and differentiability gives smooth local branch. This separation is useful in
quasilinear problems, where the natural inverse may be continuous or Lipschitz without being $C^1$.
\medskip

Distinct nondegenerate zeros of $F$ yield distinct responses by choosing disjoint neighborhoods. A convenient sufficient condition for
\textup{(L)} is $\ell(r)\to0$, which happens, for example, when
$R(x,\varepsilon)=\cR(\varepsilon(I-Q)N(x,\varepsilon))$ for which the maps on the right side are locally Lipschitz. 

\medskip

A direct result that we can prove using our approach is related on how the convergence for $\varepsilon \to 0$ is for a local solution.

\begin{corollary}[Asymptotic profile]\label{cor:profile}
Assume that $QN$ is $C^1$ near $(a_0,0)$, that \eqref{eq:ND} holds, and
that $R$ is a complementary correction. Let
$x_\varepsilon\to a_0$ be a local solution family for
$\varepsilon\to0^+$, and suppose that, locally uniformly,
\begin{equation}\label{eq:scaling}
 \varepsilon^{-\gamma}R(x,\varepsilon)\longrightarrow R_0(x),
 \qquad 0<\gamma\le1.
\end{equation}
By defining $w_0: =R_0(a_0)$ and $D:=D_x(QN)(a_0,0)$, then
\begin{equation}\label{eq:profile}
 \frac{x_\varepsilon-a_0}{\varepsilon^\gamma}\longrightarrow
 \begin{cases}
 w_0-F'(a_0)^{-1}Dw_0, & 0<\gamma<1,\\[1mm]
 w_0-F'(a_0)^{-1}\bigl(Dw_0+\partial_\varepsilon(QN)(a_0,0)\bigr),
 & \gamma=1.
 \end{cases}
\end{equation}
\end{corollary}

\begin{proof}
From \eqref{eq:T} and \eqref{eq:scaling} we can conclude that $v_\varepsilon/\varepsilon^\gamma\to w_0$. Moreover,
\[
 D_vA(0,0)=-F'(a_0)^{-1}D|_{X_2},\qquad
 \partial_\varepsilon A(0,0)
 =-F'(a_0)^{-1}\partial_\varepsilon(QN)(a_0,0).
\]
The first-order expansion of $A(v_\varepsilon,\varepsilon)$ gives
\eqref{eq:profile}; for $\gamma<1$, $\varepsilon=o(\varepsilon^\gamma)$.
\end{proof}

This corollary not only gives the convergence $x_\varepsilon\to a_0$, but also shows that if the
complementary inverse is positively homogeneous of degree $\gamma$, then the first displacement is naturally of order
$O(\varepsilon^\gamma)$. For a nonquadratic principal part this exponent need
not be one. The $p$-Laplacian example below shows this explicitly, with
$\gamma=1/(p-1)$.

\section{Applications}\label{sec:example}

\subsection[A p-Laplacian Duffing system]{A $p$-Laplacian Duffing system}
\medskip

Let $p>2$ and let
\[
    \Phi_p(\xi)=|\xi|^{p-2}\xi,\qquad \xi\in\R^n,
\]
and let $A=(a_{jk})\in\R^{n\times n}$. Consider
\begin{equation}\label{eq:duffing}
 \begin{cases}
 (\Phi_p(u'))'_j+\varepsilon\!\left[\displaystyle\sum_{k=1}^n
 a_{jk}u'_k+b_ju_j^3+c_ju_j^2+d_ju_j\right]=\varepsilon e_j(t),
 & 1\le j\le n,\\
 u(0)=u(T),\qquad u'(0)=u'(T),
 \end{cases}
\end{equation}
where $e\in C_T(\R^n)$, $b_j,c_j,d_j\in\R$, and $c_j=0$ whenever
$b_j=0$. Put
\[
    \bar e=\frac1T\int_0^T e(t)\,dt,
    \qquad
    g_j(s)=b_js^3+c_js^2+d_js.
\]

The condition
$c_j=0$ whenever $b_j=0$ simply ensures that, when the cubic term is absent,
the restoring force is linear.
\medskip

Take $X=C_T^1(\R^n)$, $Z=C_T(\R^n)$, and define
\[
    Mu=(\Phi_p(u'))',
    \qquad
    N_j(u)=e_j-\sum_{k=1}^n a_{jk}u'_k-g_j(u_j),
\]
where $M:D(M)\subset X\to Z$ has domain $(M)=\{u\in X:\Phi_p(u')\in C_T^1(\R^n)\}.$
\medskip

Then $E=\ker M=\R^n$, identified with the constant
functions. We take $Z_1=\R^n$ and define the projections
$P:X\to E$ and $Q:Z\to Z_1$ by
\[
    Pu=\bar u:=\frac1T\int_0^T u(t)\,dt,
    \qquad
    Qh=\bar h:=\frac1T\int_0^T h(t)\,dt.
\]
Consequently, the complementary spaces are
\[
    X_2=\ker P=\{u\in X:\bar u=0\},
    \qquad
    Z_2=\ker Q=\{h\in Z:\bar h=0\}.
\]

We now construct the complementary inverse.
For $h\in Z_2$, set
\[
    H_h(t)=\int_0^t h(s)\,ds.
\]
Since $\bar h=0$, the function $H_h$ is $T$-periodic. Solving
$M v=h$ gives after one integration
\[
    \Phi_p(v'(t))=H_h(t)+c
\]
for some $c\in\R^n$. Since $\Phi_p^{-1}=\Phi_{p'}$, where
$p'=p/(p-1)$, the periodicity condition for $v$ becomes
\begin{equation}\label{eq:cp}
    \int_0^T \Phi_{p'}(H_h(t)+c)\,dt=0.
\end{equation}
The map
$ c\mapsto\int_0^T \Phi_{p'}(H_h(t)+c)\,dt$ is continuous, strictly
monotone and coercive, and therefore has a unique zero $c=c(h)$ in
\eqref{eq:cp}. Defining
\[
    (\mathcal Kh)'(t)=\Phi_{p'}(H_h(t)+c(h)),
    \qquad \overline{\mathcal Kh}=0,
\]
yields the standard periodic inverse
$\mathcal K:Z_2\to X_2$ used in
\cite[Section~2]{ManasevichMawhin1998}. It is continuous and compact, and
satisfies $M\mathcal Kh=h$. Consequently,
\begin{equation}\label{eq:R-duffing}
    R(u,\varepsilon)
    =\mathcal K\bigl(\varepsilon(I-Q)N(u)\bigr)
\end{equation}
is a complementary correction satisfying \textup{(C)}, hence we can apply Theorem \ref{thm:main} (i), giving with this existence of solutions for the Duffing system.

\medskip

Notice that since $p>2$,
we have $1<p'<2$, so $\Phi_{p'}$ is not differentiable at the origin. Therefore there is no invertible linearization that can replace $\mathcal K$ near
zero.
\medskip

For $\alpha\in E$, the velocity terms vanish after averaging, and hence
\[
    F_j(\alpha)=\bar e_j-g_j(\alpha_j).
\]
The averaged equation therefore decouples into the scalar algebraic
conditions
\begin{equation}\label{eq:BE}
 b_j\alpha_j^3+c_j\alpha_j^2+d_j\alpha_j-\bar e_j=0,
 \qquad 1\le j\le n.
\end{equation}
Moreover,
\[
 F'(\alpha)=-\operatorname{diag}
 \bigl(g_1'(\alpha_1),\ldots,g_n'(\alpha_n)\bigr).
\]
Thus $\alpha$ is nondegenerate exactly when each chosen scalar root is
simple. In particular, if the $j$th equation in \eqref{eq:BE} has
$m_j\in\{1,2,3\}$ simple real roots, then the averaged map has
$m_1\cdots m_n$ distinct nondegenerate zeros.
\medskip

By the homogeneity of $\Phi_{p'}$, we have
\begin{equation}\label{eq:Khom}
 \mathcal K(\lambda h)
 =\operatorname{sgn}(\lambda)|\lambda|^{1/(p-1)}\mathcal Kh.
\end{equation}
Now, let $\alpha$ be a zero of $F$. Then
\[
    (I-Q)N(\alpha)=e-\bar e.
\]
Since $\mathcal K(e-\bar e)$ is periodic and has zero mean,
the averages of both $\mathcal K(e-\bar e)$ and its
derivative vanish. As the coefficients of the linearization
of $N$ at $\alpha$ are constant, it follows that
\[
    D_x(QN)(\alpha)\mathcal K(e-\bar e)=0.
\]
Finally, since $\gamma=1/(p-1)\in(0,1)$,
Corollary~\ref{cor:profile} and \eqref{eq:Khom} yield
the asymptotic profile in
Proposition~\ref{prop:duffing-application}.

\bigskip

\begin{proposition}[Vector $p$-Laplacian Duffing system]
\label{prop:duffing-application}
Assume $p>2$ and the hypotheses above. Suppose that, for each
$j=1,\ldots,n$, the scalar equation in \eqref{eq:BE} has
$m_j\in\{1,2,3\}$ simple real roots. Then, for all sufficiently
small $0<|\varepsilon|$, problem \eqref{eq:duffing} admits at least
\[
    m_1\cdots m_n
\]
distinct $T$-periodic solutions. Each solution approaches, in
$C_T^1(\R^n)$ as $\varepsilon\to0$, one of the corresponding
constant zeros $\alpha$ of the averaged equation.

For each such $\alpha$, the corresponding solution satisfies
\[
    u_\varepsilon
    =\alpha+\varepsilon^{1/(p-1)}
      \mathcal K(e-\bar e)
      +o\!\left(\varepsilon^{1/(p-1)}\right)
      \quad\text{in }C_T^1(\R^n),
\]
as $\varepsilon\to0^+$. If $e-\bar e\not\equiv0$,
the first-order correction is nonzero and has order
$\varepsilon^{1/(p-1)}$.
\end{proposition}

\begin{proof}
The multiplicity follows from Theorem~\ref{thm:main}(i), applied in disjoint
neighborhoods of the $m_1\cdots m_n$ nondegenerate zeros of $F$.
The asymptotic formula follows from Corollary~\ref{cor:profile} and
\eqref{eq:Khom}.
\end{proof}

\subsection{A nonsmooth equation in a Hilbert space}
\medskip

Consider the following $2\pi$-periodic problem
\begin{equation}\label{eq:Hilbert}
 (\Phi(u'))'+\varepsilon u=\varepsilon e(t),
 \qquad u(0)=u(2\pi),\quad u'(0)=u'(2\pi),
\end{equation}
with $e\in C_{2\pi}(H)$, $H=L^2(0,1)$, $0<\lambda<1$, and
$\Phi:H\to H$ the Nemytskii operator induced by
\[
    \rho(s)=s+\lambda|s|.
\]
Thus $(\Phi v)(x)=\rho(v(x))$. Since the map $\rho$ is strictly increasing
and bi-Lipschitz, but not $C^1$ in all of its domain, we can only apply  Theorem~\ref{thm:main}(ii).
\medskip

The inverse is induced by
\[
 \psi(s)=
 \begin{cases}
 s/(1+\lambda),&s\ge0,\\
 s/(1-\lambda),&s<0,
 \end{cases}
\]
and, with $\Psi=\Phi^{-1}$, one has
\begin{equation}\label{eq:H-Lip}
 \langle\Psi(\xi)-\Psi(\eta),\xi-\eta\rangle_H
 \ge\frac{1}{1+\lambda}\|\xi-\eta\|_H^2,
 \qquad
 \|\Psi(\xi)-\Psi(\eta)\|_H
 \le\frac{1}{1-\lambda}\|\xi-\eta\|_H.
\end{equation}
Hence we have strong monotonicity and Lipschitz estimate for $\Psi$.
\medskip

Set $X=C_{2\pi}^1(H)$, $Z=C_{2\pi}(H)$,
$Mu=(\Phi(u'))'$ and $N(u)=e-u$, with $P$ and $Q$ given by averages.
For $h\in Z_2$, let
\[
    H_h(t)=\int_0^t h(s)\,ds.
\]
As in the finite-dimensional $\phi$-Laplacian construction, a solution of
$Mv=h$ must satisfy
\[
    v'(t)=\Psi(H_h(t)+c)
\]
for a constant $c\in H$. The periodicity condition is therefore
\begin{equation}\label{eq:c-H}
    \int_0^{2\pi}\Psi(H_h(t)+c)\,dt=0.
\end{equation}

By \eqref{eq:H-Lip}, the operator
\[
 A_h(c):=\int_0^{2\pi}\Psi(H_h(t)+c)\,dt
\]
is continuous and strongly monotone. Indeed, for any
$c_1,c_2\in H$,
\[
 \langle A_h(c_1)-A_h(c_2),c_1-c_2\rangle_H
 \ge \frac{2\pi}{1+\lambda}\|c_1-c_2\|_H^2.
\]
In particular, $A_h$ is coercive, and the Browder--Minty
theorem \cite{Zeidler1990} gives a unique
$c=c(h)\in H$ satisfying \eqref{eq:c-H}.

We can now define $\mathcal Kh$ as the unique
zero-mean periodic function satisfying
\[
    (\mathcal Kh)'=\Psi(H_h+c(h)).
\]
By construction, $M\mathcal Kh=h$.

To see that $\mathcal K$ is Lipschitz continuous,
let $h_1,h_2\in Z_2$ and write $c_i=c(h_i)$.
Using $A_{h_i}(c_i)=0$ and \eqref{eq:H-Lip},
we obtain
\[
 \frac{2\pi}{1+\lambda}\|c_1-c_2\|_H^2
 \le \frac{1}{1-\lambda}
 \int_0^{2\pi}
 \|H_{h_1}(t)-H_{h_2}(t)\|_H
 \|c_1-c_2\|_H\,dt.
\]
Hence
\[
 \|c(h_1)-c(h_2)\|_H
 \le \frac{1+\lambda}{1-\lambda}
 \|H_{h_1}-H_{h_2}\|_\infty.
\]
Since
\[
 \|H_{h_1}-H_{h_2}\|_\infty
 \le 2\pi\|h_1-h_2\|_\infty,
\]
the Lipschitz continuity of $\Psi$ gives the same
estimate for the derivatives of
$\mathcal Kh_1$ and $\mathcal Kh_2$, up to a constant.
Finally, both functions have zero mean, so their
$C^1$-distance is controlled by the difference
of their derivatives. We conclude that
\begin{equation}\label{eq:K-H-Lip}
 \|\mathcal Kh_1-\mathcal Kh_2\|_{C^1}
 \le C_K\|h_1-h_2\|_\infty.
\end{equation}

Consequently,
\[
 R(u,\varepsilon)
 =\mathcal K\bigl(\varepsilon(I-Q)N(u)\bigr)
\]
satisfies
\[
 \|R(u,\varepsilon)-R(v,\varepsilon)\|_{C^1}
 \le C|\varepsilon|\,\|u-v\|_{C^1}.
\]
Thus the Lipschitz modulus in \textup{(L)} tends to zero with
$|\varepsilon|$. On the kernel, identified with the constant functions,
\[
    F(a)=\bar e-a,
    \qquad F'(\bar e)=-I_H.
\]
Furthermore $D_x(QN)(\bar e)|_{X_2}=0$, because $X_2$ consists of zero-mean
functions, and therefore $B_0=0$.

\bigskip

We, therefore, obtain the following result.

\medskip
\begin{proposition}\label{prop:hilbert-application}
For every $0<|\varepsilon|$ sufficiently small, problem \eqref{eq:Hilbert} has a unique $2\pi$-periodic solution in a fixed neighborhood of $\bar e$ in
$C_{2\pi}^1(H)$. Moreover,
\[
    u_\varepsilon\longrightarrow \bar e
    \qquad\text{in }C_{2\pi}^1(H)
    \quad\text{as }\varepsilon\to0.
\]
\end{proposition}

\begin{proof}
The estimate above gives
\[
 \|R(u,\varepsilon)-R(v,\varepsilon)\|_{C^1}
 \le C|\varepsilon|\,\|u-v\|_{C^1},
\]
so its Lipschitz modulus tends to zero with $|\varepsilon|$. Since
$F(\bar e)=0$, $F'(\bar e)=-I_H$, and $B_0=0$,
Theorem~\ref{thm:main}(ii) applies.
\end{proof}

\backmatter

\section*{Declarations}

\noindent\textbf{Data availability.} No datasets were generated or analyzed during this study.

\noindent\textbf{Funding.} This research did not receive any specific grant from funding agencies in the public, commercial, or not-for-profit sectors.

\noindent\textbf{Competing interests.} The authors declare that they have no known competing financial interests or personal relationships that could have appeared to influence the work reported in this paper.


\begin{thebibliography}{00}

\bibitem{BuicaLlibre2004}
A. Buic\u{a}, J. Llibre,
Averaging methods for finding periodic orbits via Brouwer degree,
\textit{Bull. Sci. Math.} 128 (2004) 7--22.
\newblock doi:10.1016/j.bulsci.2003.09.002.

\bibitem{LlibreNovaesTeixeira2014}
J. Llibre, D. D. Novaes, M. A. Teixeira,
Higher order averaging theory for finding periodic solutions via Brouwer degree,
\textit{Nonlinearity} 27 (2014) 563--583.
\newblock doi:10.1088/0951-7715/27/3/563.

\bibitem{LlibreNovaes2015}
J. Llibre, D. D. Novaes,
Improving the averaging theory for computing periodic solutions of the differential equations,
\textit{Z. Angew. Math. Phys.} 66 (2015) 1401--1412.
\newblock doi:10.1007/s00033-014-0460-3.

\bibitem{CandidoLlibreNovaes2017}
M. R. C\^{a}ndido, J. Llibre, D. D. Novaes,
Persistence of periodic solutions for higher order perturbed differential systems
via Lyapunov--Schmidt reduction,
\textit{Nonlinearity} 30 (2017) 3560--3586.
\newblock doi:10.1088/1361-6544/aa7e95.

\bibitem{CandidoNovaes2026}
M. R. C\^{a}ndido, D. D. Novaes,
On the stability of smooth branches of periodic solutions for higher order
perturbed differential equations,
\textit{Qual. Theory Dyn. Syst.} 25 (2026), Article 52.
\newblock doi:10.1007/s12346-026-01470-8.

\bibitem{Sanders2007}
J.A. Sanders, F. Verhulst, J. Murdock,
\textit{Averaging Methods in Nonlinear Dynamical Systems}, 2nd ed.,
Springer, New York, 2007.

\bibitem{CidMawhinZima2021}
J.A. Cid, J. Mawhin, M. Zima,
An abstract averaging method with applications to differential equations,
\textit{J. Differential Equations} 274 (2021) 231--250.

\bibitem{GeRen2004}
W. Ge, J. Ren,
An extension of Mawhin's continuation theorem and its application to
boundary value problems with a $p$-Laplacian,
\textit{Nonlinear Anal.} 58 (2004) 477--488.

\bibitem{ManasevichMawhin1998}
R. Man\'asevich, J. Mawhin,
Periodic solutions for nonlinear systems with $p$-Laplacian-like operators,
\textit{J. Differential Equations} 145 (1998) 367--393.

\bibitem{ManasevichMawhin2000}
R. Man\'asevich, J. Mawhin,
Boundary value problems for nonlinear perturbations of vector
$p$-Laplacian-like operators,
\textit{J. Korean Math. Soc.} 37 (2000) 665--685.

\bibitem{Mawhin2001}
J. Mawhin,
Periodic solutions of systems with $p$-Laplacian-like operators,
in: \textit{Nonlinear Analysis and its Applications to Differential Equations},
Birkh\"auser, Boston, 2001, pp. 37--63.

\bibitem{MawhinNovoaIFT}
J. Mawhin, J. Novoa,
An averaging method for periodic solutions of quasilinear equations in Banach spaces,
arXiv preprint arXiv:2610.02156 (2026).

\bibitem{RoM}
N. Rouch\'e, J. Mawhin,
\textit{Ordinary Differential Equations: Stability and Periodic Solutions},
Pitman, Boston, 1980.

\bibitem{Maw1}
J. Mawhin,
Functional analysis and boundary value problems,
in: J.K. Hale (Ed.),
\textit{Studies in Ordinary Differential Equations},
MAA Studies in Mathematics, vol. 14,
Mathematical Association of America, Washington, DC, 1977,
pp. 128--168.

\bibitem{Maw2}
J. Mawhin,
Resonance problems for some non-autonomous ordinary differential equations,
in: \textit{Stability and Bifurcation Theory for Non-Autonomous
Differential Equations},
Lecture Notes in Mathematics, vol. 2065,
Springer, Heidelberg, 2013, pp. 103--184.

\bibitem{GMMT}
M. Garc\'ia-Huidobro, R. Man\'asevich, J. Mawhin, S. Tanaka,
Periodic solutions for nonlinear systems of ODEs with generalized
variable-exponent operators,
\textit{J. Differential Equations} 388 (2024) 34--58.

\bibitem{Zeidler1990}
E. Zeidler,
\textit{Nonlinear Functional Analysis and its Applications, II/B:
Nonlinear Monotone Operators},
Springer, New York, 1990.

\end{thebibliography}
\end{document}